\documentclass[11pt]{amsart}
\usepackage{latexsym}
\usepackage{amssymb}
\usepackage{amsmath,amstext,amscd}
\usepackage{mathrsfs}
\usepackage{graphicx}
\usepackage{verbatim}
\usepackage[all]{xy}
\usepackage{amsfonts}
\usepackage{indentfirst}
\usepackage{enumerate}
\usepackage{txfonts}
\usepackage{pxfonts}
\usepackage{bezier,longtable}
\usepackage[labelformat=empty]{caption}

\newtheorem{theo}{Theorem}

\newtheorem{prop}{Proposition}

\newtheorem{lemma}[theo]{Lemma}

\newtheorem{corollary}[theo]{Corollary}

\newtheorem{conjecture}[theo]{Conjecture}
\renewenvironment{proof}{ \emph{Proof}}{$\Box$}

\newcommand{\TT}{\mathbb{T}}
\newcommand {\ZZ} {\mathbb {Z}}

\newcommand {\NN} {\mathbb {N}}

\newcommand {\Eps} {\mathcal {E}}

\newcommand{\id}{\mathrm{id}}

\renewcommand{\gg}{\mathfrak{g}}

\renewcommand {\phi} {\varphi}

\newcommand{\refle}[1]{Lemma \ref{#1}}

\newcommand{\refprop}[1]{Proposition \ref{#1}}

\newcommand{\refeq}[1]{(\ref{#1})}

\renewcommand{\Im}{\mathrm{Im}}

\newcommand{\soc}{\mathrm{soc}}
\newcommand{\Ker}{\mathrm{Ker}}

\newcommand{\rad}{\mathrm{rad}}

\newcommand{\Hom}{\mathrm{Hom}}

\newcommand{\field}[1]{\mathbb{#1}}

\def\cplus{\hbox{$\subset${\raise0.3ex\hbox{\kern -0.55em ${\scriptscriptstyle +}$}}\ }}

\def\clplus{\hbox{$\subset${\raise0.3ex\hbox{\kern -0.55em ${\scriptscriptstyle +}$}}\ }}
\def\crplus{\hbox{$\supset${\raise1.05pt\hbox{\kern -0.55em ${\scriptscriptstyle +}$}}\ }}

\title{A categorification of the boson-fermion correspondence via representation theory of $sl(\infty)$}

\author[Igor Frenkel]{\;Igor Frenkel}
\address{
Igor Frenkel
\newline Department of Mathematics
\newline Yale University 
\newline 10 Hillhouse Avenue, P.O. Box 208283
\newline New Haven, CT 06520-8283, USA}
\email{frenkel@math.yale.edu}

\author[Ivan Penkov]{\;Ivan Penkov}

\address{
Ivan Penkov
\newline School of Engineering and Science
\newline Jacobs University Bremen
\newline Campus Ring 1
\newline 28759 Bremen, Germany}
\email{i.penkov@jacobs-university.de}

\author[Vera Serganova]{\;Vera Serganova}

\address{
Vera Serganova
\newline Department of Mathematics
\newline University of California Berkeley
\newline Berkeley CA 94720, USA}
\email{serganov@math.berkeley.edu}
\begin{document}
\maketitle
%\tableofcontents
%\begin{abstract}
%\end{abstract}
\begin{abstract} In recent years different aspects of categorification of the boson-fermion correspondence
have been studied. In this paper we propose a categorification of the boson-fermion correspondence based on
the category of tensor modules of the Lie algebra $sl(\infty)$ of finitary infinite matrices. By $\mathbb T^+$ we denote
the category of ``polynomial'' tensor  $sl(\infty)$-modules. There is a natural ``creation'' functor $\mathcal T_N: \mathbb T^+\to \mathbb T^+$,
$M\mapsto N\otimes M,\quad M,N\in \mathbb T^+$. The key idea of the paper is to employ the entire category $\mathbb T$ of tensor $sl(\infty)$-modules
in order to define the ``annihilation'' functor $\mathcal D_N: \mathbb T^+\to \mathbb T^+$ corresponding to $\mathcal T_N$.
We show that the relations allowing to express fermions via bosons arise from relations in the cohomology of complexes of linear endofunctors on $\mathbb T^+$. 

 
\end{abstract}
\section {Introduction}

The origin of the boson-fermion correspondence can be traced back to the celebrated 
Jacobi triple product identity

\begin{equation}\label{Jacobi}
  \frac{\sum_{n \in \ZZ}t^nq^\frac{n^2}{2}}{\prod_{n \geq 1}(1-q^n)} = \prod_{n \geq 1} (1+tq^{n-\frac{1}{2}})(1+t^{-1}q^{n-\frac{1}{2}})
\end{equation}
which one should consider as an equality of two generating functions

\begin{equation}\label{Jacgenf}
 \sum_{n \in \ZZ, 2m \geq 0} b_{n,m} t^nq^m = \sum_{n \in \ZZ, 2m \geq 0} f_{n,m} t^nq^m\nonumber
\end{equation}
with nonnegative integral coefficients.

The boson-fermion correspondence can be viewed as a ``categorification'' of this identity.
Namely, it is an isomorphism of doubly graded vector spaces, called bosonic and fermionic 
Fock spaces,

\begin{equation}\label{vectbf}
 \bold{B} = \bigoplus_{n \in \ZZ, 2m \geq 0} \bold{B}_{n,m} \cong \bigoplus_{n \in \ZZ, 2m \geq 0} \bold{F}_{n,m} = \bold{F}
\end{equation}
such that

\begin{equation}\label{vectcons}
  \dim \bold{B}_{n,m} = b_{n,m}, \;\;\;  \dim \bold{F}_{n,m} = f_{n,m},\nonumber
\end{equation}
and their structure yields respectively the left-hand and right-hand sides of (\ref{Jacobi}). The 
identity (\ref{Jacobi}) itself follows then from the isomorphism (\ref{vectbf}). The bosonic Fock space
is naturally a representation of an infinite-dimensional Heisenberg Lie algebra, while the 
fermionic Fock space is naturally a representation of an infinite-dimensional Clifford
algebra, see section $2$ for details.

The fact that the Heisenberg Lie algebra can be constructed from the Clifford algebra and vice versa was 
first noticed in the physics literature \cite{Mn}, and was given the name boson-fermion correspondence.
It was understood in \cite{F} that the bosonic and fermionic Fock spaces are just two realizations of 
an affine Lie algebra representation. In particular, one can see the bosonic and fermionic Fock
spaces (\ref{vectbf}) as a representation of  a central extension $\widehat{sl}(\infty)$ of the
Lie algebra of infinite matrices with finitely many non-zero diagonals, \cite{DJKM}.

If we consider the subspaces $ \bold{B}_0\subset  \bold{B}$, $\bold{F}_0\subset \bold{F}$

\begin{equation}\label{sub}
 \bold{B}_0 := \bigoplus_{m \geq 0} \bold{B}_{0,m} \cong \bigoplus_{m \geq 0} \bold{F}_{0,m} =: \bold{F}_0,
\end{equation}
then the left-hand side of (\ref{Jacobi}) implies that

\begin{equation}\label{subdim}
 \dim \bold{B}_{0,m} = \dim\bold{F}_{0,m} = p(m)\nonumber
\end{equation}
where $p:\NN\to \NN$ is the partition function. The value $p(m)$ equals the size of the character table of the symmetric group $S_m$, and
the transition matrix between the bases in $\bold{B}_{0,m}$ and $\bold{F}_{0,m}$ is precisely the 
character table of $S_m$. This suggests a second categorification of the boson-fermion correspondence, more
specifically of $\bold{B}_0 \cong \bold{F}_0$, via representation theory of all symmetric groups
$\{S_n\}_{n \geq 0}$ or, by Schur-Weyl duality, via representation theory of the Lie algebra $sl(\infty):=sl(\infty,\mathbb C)$ 
of traceless finitary infinite matrices.

In this second categorification $sl(\infty)$  appears without a central extension and one is led to 
consider the category of tensor modules $\mathbb{T}^+$ whose Grothendieck ring is given by (\ref{sub}). 
The passage to the full boson-fermion correspondence $\bold{B} \cong \bold{F}$ is achieved by considering the 
derived category $\mathbb{DT}^+$.

Initially the idea of categorification was motivated by an attempt to lift three-dimensional 
invariants to four-dimensional couterparts replacing nonnegative integers by vector spaces of corresponding 
dimensions, vector spaces with bases by categories, categories by $2$-categories, etc. \cite{CF}, \cite{K1}.
It was quickly realized that the semisimple categories that give rise to three-dimensional invariants should be
generalized to more interesting abelian or even triangulated categories, to get a nontrivial categorification 
at the next level. Besides in topology, these ideas were successfully used in algebraic geometry and 
representation theory and have led to many interesting examples of categorification in the last 15 years.

In our example, both $\mathbb{T}^+$ and $\mathbb{DT}^+$ are semisimple categories and, in order to obtain
a nontrivial categorification, one should look for non-semisimple generalizations. Luckily, there is a 
natural category of tensor modules, denoted by $\mathbb{T}$ \cite{DPS}, based on the vector representation 
$V$ of $sl(\infty)$ and its restricted dual $V_*$. Unlike its finite-dimensional analogue, $\mathbb{T}$ is not a
semisimple category, see section $3$ for details. The starting point of our categorification of the boson-fermion 
correspondence is the projection functor

\begin{equation}\label{projfunct}
 (\mbox{} \cdot \mbox{})^+:\mathbb{T}\rightarrow \mathbb{T}^+.\nonumber
\end{equation}
This allows us to define, besides the obvious ``creation functor''

\begin{equation}\label{funct}
 \mathcal{T}_N : \mathbb{T}^+ \rightarrow \mathbb{T}^+, \;\;\; M \rightarrow N \otimes M,\;\; N,M \in \mathbb{T}^+,\nonumber
\end{equation}
its important counterpart, the ``annihilation functor''

\begin{equation}\label{anfunct}
 \mathcal{D}_N : \mathbb{T}^+ \rightarrow \mathbb{T}^+, \;\;\; M \rightarrow (N_* \otimes M)^+,\;\; N,M \in \mathbb{T}^+.\nonumber
\end{equation}
We show in section $4$ that $ \mathcal{D}_N$ is left adjoint to $\mathcal{T}_N$.

Together, the two functors $\mathcal{T}_N$ and $\mathcal{D}_N$  generate an abelian subcategory $\widehat{\mathbb{T}}$ of the category
of linear endofunctors on $\TT^+$, and the category $\widehat{\mathbb{T}}$ can be viewed as a realization of $\mathbb T$. 
In particular, the simplest nontrivial exact sequence in ${\mathbb{T}}$
\begin{equation}\label{Halgebra1}
 0 \rightarrow sl(V)\rightarrow V_*\otimes V \rightarrow \mathbb C \rightarrow 0\nonumber
\end{equation}
gives rise to a categorification of the simplest nontrivial  
relation in the Heisenberg algebra acting on $\bold{B}$:

\begin{equation}\label{Halgebra}
 0 \rightarrow \mathcal{T}_V \circ \mathcal{D}_V \rightarrow \mathcal{D}_V \circ \mathcal{T}_V \rightarrow Id \rightarrow 0.
\end{equation}

To categorify the boson-fermion correspondence one needs to consider the creation and annihilation functors corresponding 
to the extreme partitions of $n$: $n$ itself and $1_n = \underbrace{(1,...,1)}_{\text{n}}$. We denote

\begin{equation}\label{notfunct}
 \mathcal{H}_n:=\mathcal{T}_{S^n(V)},\;\; \Eps_n:=\mathcal{T}_{\wedge^{n}(V)},\;\;\;\; \mathcal{H}^*_n:=\mathcal{D}_{S^n(V)},\;\; \Eps^*_n:=\mathcal{D}_{\wedge^n(V)}.
\end{equation}
The exact sequences between these functors yield a categorification of certain identities that appear in the boson-fermion
correspondence, see section $5$.

The creation and annihilation functors (\ref{notfunct}) and the relations between them provide the building blocks for the
categorifications of the operators in the Heisenberg and Clifford algebras. However, to define the corresponding  functors
we need to work with the category of complexes in $\widehat{\mathbb{T}}$. In section 6 we introduce complexes of functors

\begin{equation}\label{note1}
\dots \rightarrow \mathcal{H}_{p+1} \circ \Eps_{q+1}^* \rightarrow \mathcal{H}_{p} \circ \Eps_{q}^* \rightarrow \mathcal{H}_{p-1} \circ \Eps_{q-1}^* \dots, \nonumber
\end{equation}

\begin{equation}\label{note2}
 \dots\rightarrow \Eps_{p+1} \circ \mathcal{H}_{q+1}^* \rightarrow \Eps_{p} \circ \mathcal{H}_{q}^*  \rightarrow \Eps_{p-1} \circ \mathcal{H}_{q-1}^* \dots \nonumber 
\end{equation}
for  $p-q=a$, which we denote by  $\mathcal{X}_a$, and $\mathcal{X}_a^*$, respectively.
Then in section $7$ we verify that the commutation relations of these complexes of functors yield, at the Grothendieck ring level,
the usual relations between the generators of the Clifford algebra.

Finally, in section $8$ we introduce the complexes of functors $\mathcal{P}_k$ and $\mathcal{P}_k^*$ and verify that their commutation
relations categorify the familiar relations between the generators of the Heisenberg algebra.

Different aspects of categorification of the boson-fermion correspondence have been studied by several authors. A combinatorial 
version of the categorification of the Heisenberg algebra using a counterpart of the functors (\ref{notfunct}) is obtained
in \cite{K2}. It has been further extended in \cite{LS1} and a survey of related works can be found in \cite{LS2}. Another 
categorification of the bosonic Fock space and the action of the Lie algebra $\widehat{sl}(\infty)$ 
has been obtained in a series of papers \cite{HY1}, \cite{HTY}, \cite{HY2}. This approach is 
closest to ours, though our tensor representations of $sl(\infty)$ are ``rational'' (not ``polynomial'') which leads to a non-semisimple 
category. In a more general setting, the authors of \cite{CL1}, \cite{CL2} have constructed
a categorification of the basic representation of the affine Lie algebras using the derived categories of coherent
sheaves on Hilbert schemes of points on ALE spaces. For the simplest ALE space $\mathbb{C}^2$ their construction should yield
a categorification of the boson-fermion correspondence. It is a very interesting problem to obtain these more general 
geometric categories as representation categories for appropriate generalizations of $sl(\infty)$. Extending the 
existing terminology of categorification it would be natural to call a solution to this problem --- representification ---
namely, a realization of a geometric category as a certain representation category associated with a particular Lie algebra. 
Strictly speaking, one still has to verify that our representation-theoretic categorification of boson-fermion correspondence is equivalent to the earlier
geometric categorification. If so, one can view  our present paper as a first example of representification of the categories of sheaves on Hilbert
schemes of points on $\mathbb{C}^2$.  

{\bf Acknowledgements.} I. P. and V. S. acknowledge partial support from the Max-Planck Institute for Mathematics in Bonn (fall 2012) as well as continued 
partial support through the DFG priority program ``Representation Theory'' (2011--2014).  I. F. and V. S. have been supported by NSF grants  DMS-1001633
and DMS - 1303301, respectively.

\section {Recollection of the boson-fermion correspondence}

Recall that the
boson-fermion correspondence relates the actions of the
infinite Heisenberg and Clifford algebras on the Fock space, see for instance \cite{F}
or \cite{DJKM}. 

The ground field is $\mathbb C$.
Let $\bold{Cf}$ be the infinite-dimensional Clifford algebra over
with generators
$\{\psi_i,\psi^*_i\}_{i\in\mathbb Z}$ and relations 
\begin{equation}\label{clifford}
\psi_i\psi_j+\psi_j\psi_i=\psi^*_i\psi^*_j+\psi^*_j\psi^*_i=0,\,\,\psi_i\psi^*_j+\psi^*_j\psi_i=\delta_{ij}.
\end{equation}
The {\it Fock space} $\bold{F}$ can be defined as the induced module
$\bold{Cf}\otimes_{\bold{Cf}^+}\mathbb Z$, where $\bold{Cf}^+$ is the subalgebra generated
by $\psi_i$ for $i\geq 0$ and $\psi^*_i$ for $i<0$. The Fock space  is an
irreducible $\bold{Cf}$-module and is equipped with the grading
$\bold{F}=\bigoplus_{m\in\mathbb Z}\bold{F}(m)$ induced by the grading on
$\bold{Cf}$ given by deg $\psi_i=1$, deg $\psi_i^*=-1$.

Let $\bold {H}$ denote the (infinitely generated) Weyl algebra
with generators $\{p_n,p^*_n\}_{n\geq 1}$ and relations
\begin{equation}\label{boson}
[p_n,p_m]=[p_n^*,p_m^*]=0,\,\,[p_n,p^*_m]=n\delta_{mn}.
\end{equation}
The notation $\bold {H}$ reflects the fact that the Lie algebra defined by the relations (\ref{boson}) is the infinite-dimensional
Heisenberg algebra.
One can define an action of  $\bold {H}$ on $\bold{F}$ by setting 
$$p_n=\sum_{i\in \mathbb Z}\psi_i\psi^*_{i+n},\,\,p^*_n=\sum_{i\in \mathbb Z}\psi_i\psi^*_{i-n}.$$
Note that $\bold{H}$ has two commutative subalgebras: 
$\bold {H}^+$, generated by $p_n$ for $n\geq 1$, and $\bold {H}^+_*$, generated by $p_n^*$ for $n\geq 1$.

We identify $\bold{H}^+$ with the algebra  
of symmetric functions of infinitely many variables,
$$\bold{H}^+=\bigoplus_{\lambda\in\bold{Part}}\mathbb{C} s_\lambda\cong\mathbb{C}[p_1,p_2,...]=\mathbb{C}[h_1,h_2,...]=\mathbb{C}[e_1,e_2,...].$$
Here $\bold{Part}$ is the set of all partitions, $s_\lambda$ are the
Schur functions, $p_i$-s are the sums of powers, $e_i$-s are the
elementary symmetric functions, $h_i$-s are the sums of all monomials of
degree $i$. Recall  \cite{M} that $\{h_i\}_{i\geq 1},\{e_i\}_{i\geq 1}$ 
are expressed in terms of  $\{p_i\}_{i\geq 1}$ as follows
\begin{equation}\label{symid}
H(z)\stackrel{def}{=}\sum_{n\geq 0}h_n z^n=\exp\left(\sum_{n\geq 1}\frac{p_n z^n}{n}\right),\;
E(z)\stackrel{def}{=}\sum_{n\geq 0}e_n z^n=\exp\left(\sum_{n\geq  1}(-1)^{n-1}\frac{p_n z^n}{n}\right),
\end{equation}
where we set $h_0:=1$ and $e_0:=1$. 

In a similar way   
$$\bold{H}^+_*\cong\mathbb{C}[p^*_1,p^*_2,...]\cong\mathbb{C}[h^*_1,h^*_2,...]\cong\mathbb{C}[e^*_1,e^*_2,...],$$
and 
\begin{equation}\label{symid*}
H^*(z)\stackrel{def}{=}\sum_{n\geq 0}h^*_n z^{-n}=\exp\left(\sum_{n\geq 1}\frac{p^*_n z^{-n}}{n}\right),\;
E^*(z)\stackrel{def}{=}\sum_{n\geq 0}e^*_n z^{-n}=\exp\left(\sum_{n\geq 1}(-1)^{n-1}\frac{p^*_n z^{-n}}{n}\right);
\end{equation}
here again $h^*_0:=1$ and $e^*_0:=1$. Note that $H^*(z),E^*(z)$ 
contain only non-positive powers of $z$.

It is possible to show \cite{DJKM} that for each $m\in\mathbb Z$
there is an isomorphism of $\bold{H}$-modules 
$$\bold F(m)\simeq\bold H\otimes_{\bold H_*^+}\mathbb C.$$ 
In particular, each $\bold{F}(m)$ is an irreducible
$\bold{H}$-module. Moreover, as an $\bold{H}^+$-module $\bold{F}(m)$ 
can be identified with $\bold{H}^+$.
In what follows we consider $h_n,e_n,p_n,h^*_n,e_n^*$ and $p_n^*$ as 
linear operators on $\bold{H}^+$ and the identities
(\ref{symid}) and (\ref{symid*}) as  identities relating these
linear operators. Furthermore, it is a known fact that $h_n^*$ (respectively, $e_n^*$) are operators dual to
$h_n$ (respectively, $e_n$)  with respect  to the natural inner product on $\bold{H}^+$ for which
$\{s_\lambda\}_{\lambda\in\bold{Part}}$ forms an orthonormal basis.

Set 
\begin{equation}\label{eqvert}
X(z)=H(z)E^*(-z),\,\,X^*(z)=E(-z)H^*(z).
\end{equation}
Roughly speaking, up to a shift, the coefficients of $X(z)$ and $X^*(z)$ are the
fermions $\psi_i$ and $\psi_i^*$.
More precisely, 
there exists  an automorphism $s$ of the $\bold {H}$-module $\bold{F}$ such that 
$s(\bold{F}(m))= \bold{F}(m+1)$ and the restrictions of $\psi_i$ and $\psi_i^*$ on
$\bold{F}(m)$ are recovered by the formulas
\begin{equation}\label{bf}
\sum_{i\in\mathbb Z}\psi_iz^i=sz^{m}X(z),\,\, \sum_{i\in\mathbb Z}\psi^*_iz^i=s^{-1}z^{-m}X^*(z).
\end{equation}

To prove (\ref{bf}) one has to show that the Fourier coefficients of the vertex operators
$X(z)$ and $X^*(z)$ satisfy a version of the Clifford relations  (\ref{clifford}). It is known \cite{F}
that (\ref{clifford}) is equivalent to the following vertex operator identities
\begin{equation}\label{v1}
X(z)X(w)+\frac{w}{z} X(w)X(z)=0,
\end{equation}
\begin{equation}\label{v2}
X^*(z)X^*(w)+\frac{w}{z} X^*(w)X^*(z)=0,
\end{equation}
\begin{equation}\label{v3}
X(z)X^*(w)+\frac{z}{w}X^*(w)X(z)=\sum_{n\in\mathbb Z}\frac{z^n}{w^n}.
\end{equation}
In section 7 we provide a categorical proof of these identities by showing
that they are corollaries of certain relations between endofunctors on a category
of $sl(\infty)$-modules.

\section{The category $\mathbb T$ and the projection functor}

Denote by $V$ and $V_*$ a pair of countable dimensional vector spaces
in perfect duality, i.e. with a fixed non-degenerate linear map $\langle\cdot,\cdot\rangle : V\times V_*\to\mathbb C$.
As proved by G. Mackey \cite{Mk}, all such triples $(V,V_*,\langle\cdot,\cdot\rangle)$ are
isomorphic. The vector space $V\otimes V_*$ is naturally endowed with
the structure of an
associative algebra such that $$(v\otimes w)\cdot (v'\otimes w')=\langle v'\otimes w\rangle v\otimes w',$$ and hence also with a Lie algebra structure.
It is easy to check that this Lie algebra is isomorphic to the Lie algebra
$gl(\infty)$ which by definition is the Lie algebra of
infinite matrices $(a_{ij})_{i,j\in\mathbb Z}$ with finitely many non-zero entries. The subalgebra $\gg=\Ker \langle\cdot,\cdot\rangle$ is isomorphic to
$sl(\infty)$, i.e. to the subalgebra of traceless matrices
in  $gl(\infty)$. 

In \cite{DPS} a category $\mathbb T$ of tensor
$\gg$-modules has been introduced. More precisely, $\mathbb T$ is the
category of finite-length submodules (equivalently, finite-length
subquotients) of direct sums of copies of the tensor algebra
$T^\cdot(V\oplus V_*)$ considered as a $\gg$-module. It is also possible to define this category intrinsically,
and in fact this is done in three different ways in \cite{DPS}. Note that
$\mathbb T$ is a monoidal category with respect to tensor product of $\gg$-modules.

Consider the tensor algebra $T^\cdot (V)$. By Schur--Weyl duality,
$$T^\cdot (V)=\bigoplus_{\lambda\in\bold{Part}}V_\lambda\otimes A^\lambda$$
where $V_\lambda$ is the image of a Young projector corresponding to
$\lambda$ and $A^\lambda$ stands for the irreducible
representation of the symmetric group $\field{S}_{|\lambda|}$ corresponding to
the partition $\lambda$; here $|\lambda|$ is the degree of $\lambda$,
i.e. $|\lambda|=\sum_{i}\lambda_i$, where
$\lambda=\{\lambda_i\}$. Note that each $V_\lambda$ is an irreducible
$\gg$-module. Similarly,
$$T^\cdot (V_*)=\bigoplus_{\lambda\in\bold{Part}}(V_*)_\lambda\otimes A^\lambda.$$

The tensor algebra $T^\cdot(V\oplus V_*)$ is not a semisimple
$\gg$-module. More precisely, the $\gg$-module $T^{m,n}=V^{\otimes m}\otimes V_*^{\otimes n}$
is $\gg$-semisimple if and only if $mn=0$. In \cite{PStyr} the socle
filtration of $T^{m,n}$ is described as follows. 
Let $\Phi_{i_1...i_k|j_1...j_k}: T^{m,n}\rightarrow T^{m-k,n-k}$
be the contraction map given by 
$$\Phi_{i_1...i_k|j_1...j_k}(v_1\otimes ...\otimes v_m\otimes w_1\otimes ...\otimes w_n)=\langle v_{i_1}, w_{j_1} \rangle ...\langle v_{i_k}, w_{j_k} \rangle v_1\otimes ...\otimes\hat{v}_{i_1}\otimes ...\otimes\hat{v}_{i_k}\otimes ...\otimes v_m\otimes w_1\otimes ...\otimes\hat{w}_{j_1}\otimes ...\otimes\hat{w}_{j_k}\otimes ...\otimes w_n.$$
Then the socle filtration of $T^{m,n}$ is given by
\begin{equation}\label{socker}
\soc^{k}(T^{m,n})=\bigcap_{i_1...i_k|j_1...j_k}\Ker(\Phi_{i_1...i_k|j_1...j_k}).
\end{equation}
Here $k=1,\dots,\operatorname{min}(m,n)+1$ and we use the convention $\soc=\soc^1$.

The modules $V_\lambda\otimes(V_*)_\mu$ are indecomposable and represent
the isomorphism classes of indecomposable direct summands in
$T^\cdot(V\oplus V_*)$. Moreover, using (\ref{socker}) it is shown in
\cite{PStyr} that $V_\lambda\otimes (V_*)_\mu$ has a simple
socle. Denote this simple module by $V_{\lambda,\mu}$. Then, for variable $\lambda$ and $\mu$,
$V_{\lambda,\mu}$ exhaust all (up to isomorphism) simple modules of
$\TT$. In this
way, the simple objects of $\mathbb T$ are labeled by pairs of
partitions $\lambda,\mu$. Note that 
$V_\lambda=V_{\lambda,\emptyset}$ and $(V_*)_\lambda=V_{\emptyset,\lambda}$.

Furthermore, $V_\lambda\otimes (V_*)_\mu$  is an injective hull of
$V_{\lambda,\mu}$. It is shown in \cite{DPS} that any indecomposable
injective object of $\TT$ is isomorphic to $V_\lambda\otimes (V_*)_\mu$ for
some $\lambda,\mu$.
For the layers of the socle filtration of $V_\lambda\otimes (V_*)_\mu$ we have
\begin{equation}\label{socmult}
[\soc^{k+1}(V_\lambda\otimes  (V_*)_\mu)/\soc^{k}(V_\lambda\otimes(V_*)_\mu):V_{\lambda',\mu'}]=
\sum_{\gamma\in\bold{Part},\;|\;\gamma|=k}N^\lambda_{\lambda',\gamma}N^\mu_{\mu',\gamma}
\end{equation}
where $N^\nu_{\nu',\gamma}$ denote the Littlewood--Richardson
coefficients \cite{PStyr}. 
\begin{lemma}\label{dimhom} Let $|\lambda|-|\lambda'|=|\mu|-|\mu'|=1$.
Then 
$$\operatorname{dim}\Hom(V_\lambda\otimes (V_*)_\mu,V_{\lambda'}\otimes (V_*)_{\mu'})\leq 1,$$
and 
$$\Hom(V_\lambda\otimes (V_*)_\mu,V_{\lambda'}\otimes(V_*)_{\mu'})\simeq\mathbb C$$
if and only if $\lambda'$ and $\mu'$ are obtained from $\lambda$ and
$\mu$ respectively by removing one box.
\end{lemma}
\begin{proof} The statement follows from (\ref{socmult}) and the 
  injectivity of $V_{\lambda'}\otimes (V_*)_{\mu'}$. 
\end{proof}

Twisting by an outer involution of $\gg$ yields an involution (equivalence
of monoidal categories whose square is the identity
functor) $$(\;\cdot\;)_*: \mathbb T\to\mathbb T$$ which maps
$V_{\lambda,\mu}$ to $V_{\mu,\lambda}$.

By $\field{T}^+$ we denote the full semisimple subcategory
of $\field{T}$ consisting of $\mathfrak g$-modules whose simple constituents are isomorphic to $V_\lambda$ for
$\lambda\in\bold{Part}$.
Note that $V_\lambda$ is injective as an object of $\mathbb T$, hence
any object of $\mathbb T^+$ is injective in $\mathbb T$.

By $\mathbb {DT}^+$ we denote the derived category of $\mathbb T^+$. As $\mathbb T^+$ is semisimple, 
$\mathbb {DT}^+$ is semisimple, and its simple objects $V_\lambda[n]$
are labeled by pairs $\lambda\in\bold{Part},n\in\mathbb Z$.

If $M$ is an object of $\mathbb T$, $\mathbb T^+$ or $\mathbb {DT}^+$, then $[M]$ denotes the class of $M$ in the corresponding Grothendieck ring;
furthermore, $\mathcal K({\cdot})$ stands for complexified Grothendieck ring.
\begin{lemma}\label{grfock} The map $[V_\lambda]\mapsto s_\lambda$ extends to an isomorphsim $\operatorname{ch}:\mathcal K(\mathbb T^+)\to \bold{H}^+$.
\end{lemma} 
\begin{proof} The character of any representation in $\mathbb T^+$ is a symmetric function on the diagonal subalgebra
of $gl(\infty)$. By definition, the homomorphism $\operatorname{ch}$ assigns to any element of the complexified Grothendieck ring 
the corresponding linear combination of characters of modules. It is well known that $\operatorname{ch}([V_\lambda])=s_\lambda$.
Since $\{s_\lambda\}_{\lambda\in\bold{Part}}$ is a basis in $\bold{H}^+$, $\operatorname{ch}$ is an isomorphism.
\end{proof}

We denote by $\operatorname{Sch}$ the map from $\mathbb T^+\to \bold{H}^+$ defined by $\operatorname{Sch}(\cdot):=\operatorname{ch}([\cdot])$.
For a given exact functor $\mathcal F:\mathbb T^+\to\mathbb T^+$ there
exists a unique linear operator $[\mathcal F]:\bold{H}^+\to\bold{H}^+$
such that $[\mathcal F]\circ\operatorname{Sch}=\operatorname{Sch}\circ \mathcal F$.

Define a functor $(\;)^+:\field{T}\rightarrow\field{T}^+$ by setting
$$M^+:=M/(\bigcap_{\phi\in\Hom_\gg(M,T^\cdot(V))}\Ker\phi)$$ for
$M\in\mathbb T$.

\begin{lemma}\label{lemma1}
Let $M_{gr}$ denote the semisimplification of $M\in\field{T}$. Then $(M_{gr})^+\simeq M^+$.
\end{lemma}

\begin{proof}
Since $T^\cdot(V)$ is semisimple, $\rad M\subset\bigcap_{\phi\in\Hom_\gg(M,T(V))}\Ker\phi$, 
where $\rad M$ stands for the radical of $M$ considered as a $\gg$-module.
Therefore $M^+\simeq(M/\rad M)^+$. On the other hand, the Jordan-H\"{o}lder multiplicity
of $V_\lambda$ in $\rad M$ is zero for any partition $\lambda$, as $V_\lambda$ is injective
and simple. Therefore $(\rad M)^+=((\rad M)_{gr})^+=0$. This shows that 
$$(M_{gr})^+\simeq((M/\rad M)\oplus(\rad M)_{gr})^+\simeq(M/\rad M)^+\oplus((\rad M)_{gr})^+ = (M/\rad M)^+\simeq M^+.$$
\end{proof}

\begin{corollary}\label{cor2}
$(\;)^+:\field T\to \field T^+$ is an exact functor.
\end{corollary}

\section{The functors $\mathcal D_N$ and $\mathcal T_N$}

Let $N\in\mathbb T$, then $\mathcal T_N(\cdot):=(N\otimes \cdot)^+$ is an exact functor $\mathbb T^+\to\mathbb T^+$. 
If $\mathcal{E}nd_l(\mathbb T^+)$ denotes the category of all linear
endofunctors of $\mathbb T^+$ (i.e. all functors from $\TT^+$ to $\TT^+$ which induce linear operators on Homs), 
then $\mathcal T:\TT\to\mathcal{E}nd_l(\TT^+)$ is a faithful functor.
In particular, for any $M,N\in\TT$, a morphism $\varphi\in\Hom_\gg(M,N)$ induces a morphism of functors 
$$\mathcal T_{\varphi}:\mathcal T_M\to\mathcal T_N,$$
$$\mathcal T_{\varphi}(X):=(\varphi\otimes\id)^+:(M\otimes X)^+\to(N\otimes X)^+$$ 
for $X\in\TT^+$.
In particular, any exact sequence in $\TT$ 
$$0\to N\to M \to L\to 0$$
induces an exact sequence of linear endofunctors on $\TT^+$
$$0\to \mathcal T_N\to \mathcal T_M \to \mathcal T_L\to 0.$$

In what follows we use the notation $\mathcal D_N:=\mathcal T_{N_*}$.

\begin{lemma}\label{adjoint} If $N\in\TT^+$, then $\mathcal D_N$ is left adjoint to $\mathcal T_N$.
\end{lemma}
\begin{proof} We have to construct a canonical isomorphism
\begin{equation}\label{eqadj2}
\Hom_\gg((N_*\otimes X)^+,Y))\xrightarrow{\sim} \Hom_\gg(X,N\otimes Y)
\end{equation}
for any $X,Y\in\TT^+$. 
First, from the definition of ${\cdot}^+$ we have a canonical isomorphsim
\begin{equation}\label{eqadj1}
\Hom_\gg(N_*\otimes X,Y))\xrightarrow{\sim}\Hom_\gg((N_*\otimes X)^+,Y).
\end{equation}
Next, define a morphsim 
$$\gamma: \Hom_\gg(X,N\otimes Y)\to \Hom_\gg(N_*\otimes X,Y)$$
by setting $\gamma(\varphi)(m\otimes x)=\sum_i \langle n_i,m\rangle y_i$
if $\varphi(x)=\sum_i n_i\otimes y_i$.
By construction $\gamma$ is injective. It remains to show that
\begin{equation}\label{eqadj}
\dim\Hom_\gg(X,N\otimes Y)= \dim\Hom_\gg(N_*\otimes X,Y).
\end{equation}
Since $\TT^+$ is semisimple, it suffices to check (\ref{eqadj}) for
simple $N=V_\mu$, $X=V_\lambda$ and $Y=V_\nu$. By the very definition of the Littlewood-Richardson coefficients, the left-hand side equals
$N^{\lambda}_{\mu,\nu}$. The equality (\ref{socmult}) implies that the right-hand side is given by the same Littlewood-Richardson coefficient.
Therefore $\gamma$ is an isomorphism and the desired isomorphism (\ref{eqadj2}) is the composition of (\ref{eqadj1}) with $\gamma^{-1}$.
\end{proof}

\begin{corollary}\label{coradj} If $N\in\TT^+$, then $[\mathcal D_N]$ and $[\mathcal T_N]$ are mutually dual operators on $\bold{H}^+$.
\end{corollary}

It is clear that for all $L,N\in\TT^+$ we have isomorphisms of functors
$$\mathcal T_N\circ\mathcal T_L\simeq \mathcal T_{N\otimes L}$$
and 
$$\mathcal D_L\circ \mathcal T_N\simeq \mathcal T_{L_*\otimes N}.$$

\begin{lemma}\label{lemma3} For $L,N\in\mathbb T^+$ there is an isomorphism of functors 
$\mathcal D_N\circ \mathcal D_L\simeq \mathcal D_{N\otimes L}$.
\end{lemma}

\begin{proof}
Consider the natural exact sequence
$$0\rightarrow K\rightarrow L_*\otimes M\rightarrow(L_*\otimes M)^+\rightarrow 0,$$
where $M\in\mathbb T^+$.
It implies the exact sequence
$$0\rightarrow N_*\otimes K\rightarrow N_*\otimes L_*\otimes M\rightarrow N_*\otimes(L_*\otimes M)^+\rightarrow 0.$$
By \refle{lemma1} $K^+=0$ and therefore, again by \refle{lemma1}, $(N_*\otimes K)^+=0$.
Since $(\cdot)^+$ is an exact functor, we obtain a natural isomorphism
$$ \mathcal D_{N\otimes L}(M)=(N_*\otimes L_*\otimes M)^+\simeq(N_*\otimes(L_*\otimes M)^+)^+=\mathcal D_N\circ \mathcal D_L(M).$$
\end{proof}

Under a direct sum of linear operators $L_i: A\rightarrow B_i$ we understand the operator
$\oplus L_i:A\rightarrow\bigoplus_i B_i.$

\begin{lemma}\label{lemma4}
\begin{enumerate}[(a)]
 \item $(T^{m,n})^+=0$ if $m<n$.
 \item If $m\geq n$, then $(T^{m,n})^+\simeq\Im\left(\bigoplus_{i_1...i_n}\Phi_{i_1...i_n|1...n}\right)$.
 \item If $M$ is a submodule of $T^{m,n}$, then $M^+\simeq\left(\bigoplus_{i_1...i_n}\Phi_{i_1...i_n|1...n}\right)(M)$.
\end{enumerate}
\end{lemma}

\begin{proof}
By (\ref{socmult}) all simple constituents of $\soc^{k+1}(T^{m,n})/(\soc^{k}(T^{m,n}))$ are isomorphic to
$V_{\lambda,\mu}$ with $|\lambda|=n-k$, $|\mu|=m-k$. Therefore (a) follows from
\refle{lemma1}. 

In the case $n\geq m$ we have $(T^{m,n})^+\simeq T^{m,n}/(\soc^{n+1}(T^{m,n}))$.
Since $\soc^{n+1}(T^{m,n})=\bigcap_{i_1...i_n}\Ker(\Phi_{i_1...i_n|1...n})$, statement (b) follows. 

To prove (c) note that $T^\cdot(V)$ is an injective object of $\mathbb T$. Therefore 
any $\phi\in\Hom_\gg(M,T^\cdot(V))$ extends to $\tilde{\phi}\in\Hom_\gg(T^{m,n},T^\cdot(V))$,
$\tilde{\phi}|_M=\phi$. This shows that $\displaystyle M^+\simeq M/(M\cap(\bigcap_{i_1\dots i_n}\Ker\Phi_{i_1...i_n|1...n}))$,
and (c) follows.
\end{proof}

\begin{prop}\label{prop1} For any $L,N\in\TT^+$ there is an isomorphism of functors
$$\mathcal T_{\soc(L_*\otimes N)}\simeq  \mathcal T_N\circ\mathcal D_L.$$
\end{prop}

\begin{proof} One has to prove that
for any $M\in\field{T}^+$ there exists a canonical isomorphsim
$$(\soc( L_*\otimes N)\otimes M)^+\simeq  \mathcal T_N\circ\mathcal D_L(M).$$
Without loss of generality we may assume that $L\subset V^{\otimes  l}$, $M\subset V^{\otimes m}$, $N\subset V^{\otimes n}$.
By \refle{lemma4}, 
$$(\soc(L_*\otimes N)\otimes M)^+=\left(\bigoplus_{1\leq j_1,...,j_l\leq m+n}\Phi_{1...l|j_1...j_l}\right)(\soc(L_*\otimes N)\otimes M).$$

If $j_i\leq n$ at least for one $j_i$, then $\Phi_{1...l|j_1...j_l}(\soc(L_*\otimes N)\otimes M)=0$.
Therefore 
$$(\soc(L_*\otimes N)\otimes M)^+=\left(\bigoplus_{n< j_1,...,j_l\leq m+n}\Phi_{1...l|j_1...j_l}\right)(\soc(L_*\otimes N)\otimes M).$$
We will prove that  
\begin{equation}\label{contrid}
\left(\bigoplus_{n<j_1,...,j_l\leq
    m+n}\Phi_{1...l|j_1...j_l}\right)((\soc(L_*\otimes N)\otimes
M)=\left(\bigoplus_{n<j_1,...,j_l\leq
    m+n}\Phi_{1...l|j_1...j_l}\right)(L_*\otimes N\otimes M).
\end{equation}
Note that the left-hand side of (\ref{contrid}) is a submodule of the right-hand side. 

We use the fact that  
$$U:=\{w_1\otimes...\otimes w_n\otimes x_1\otimes ...\otimes x_m | w_1,...,w_n,x_1,\dots,x_m\in V, \operatorname{span} (w_1,...,w_n)\cap \operatorname{span} (x_1,\dots,x_m)=\{0\}\}$$
spans $T^{0,m+n}$. Let
$u=v_1\otimes ...\otimes v_l\otimes w_1\otimes ...\otimes w_n\otimes x_1\otimes ...\otimes x_m\in T^{l,m+n}$.
Set
$$\pi(u)=\pi_{L_*}(v_1\otimes ...\otimes v_l)\otimes\pi_N( w_1\otimes ...\otimes w_n)\otimes\pi_M(x_1\otimes ...\otimes x_m),$$
where $\pi_{L_*}: V_*^{\otimes l}\to L_*$, $\pi_N: V^{\otimes n}\to N$ and $\pi_M: V^{\otimes m}\to M$ are respective projectors.
To prove (\ref{contrid}) it is sufficient to show that for any $u$
such that $\operatorname{span} (w_1,...,w_n)\cap \operatorname{span} (x_1,\dots,x_m)=\{0\}$
there exists 
$\tilde{u}\in\soc (T^{l,n})\otimes V^{\otimes  m}$
such that 
\begin{equation}\label{eqs1}
\Phi_{1...l|j_1...j_l}(\pi(u))=\Phi_{1...l|j_1...j_l}(\pi(\tilde{u})) 
\end{equation}
for any choice of $j_1...j_l,\; n<j_1,...,j_l\leq m+n$.

For this purpose consider $\tilde{u}=\tilde{v}_1\otimes\dots\otimes\tilde{v}_l\otimes w_1\otimes...\otimes w_n\otimes x_1\otimes ...\otimes x_m $, 
where $\tilde{v}_i\in V_*$ satisfy the conditions 
$\langle \tilde{v}_i,w_j \rangle=0$
for $1\leq i\leq l,\; 1\leq j\leq n$, and $\langle \tilde{v}_i,x_j \rangle=\langle v_i,x_j \rangle$
for $1\leq i\leq l,\; 1\leq j\leq m$. Such a choice of $\tilde{v}_i$ is
possible since  $\operatorname{span} (w_1,...,w_n)\cap \operatorname{span} (x_1,\dots,x_m)=\{0\}$. 
It is a direct computation to verify that $\tilde{u}$ satisfies \refeq{eqs1}.

To finish the proof, note that  Lemma \ref {lemma4} (c) implies
$$\left(\bigoplus_{n<j_1,...,j_l\leq  m+n}\Phi_{1...l|j_1...j_l}\right)(L_*\otimes N\otimes M)\simeq 
N\otimes \left(\left(\bigoplus_{1\leq j_1,...,j_l\leq m}\Phi_{1...l|j_1...j_l} \right)(L_*\otimes M)\right)=
\mathcal T_{N}\circ \mathcal D_{L}(M).$$
\end{proof}

It is an interesting problem to characterize the image $\hat{\mathbb T}$ of $\mathbb T$ inside the category of linear endofunctors
$\mathcal{E}nd_l(\mathbb T^+)$. 

\section{Categorifying identities for halfs of vertex operators}

As a preliminary step to  categorifying the identites (\ref{v1})--(\ref{v3}),
in this section we categorify the following identities:
\begin{equation}\label{equ1}
H(z)E(-z)=E(-z)H(z)=1,
\end{equation}
\begin{equation}\label{equ2}
 H^*(z)E^*(-z)=E^*(-z)H^*(z)=1,
\end{equation}
\begin{equation}\label{equ3}
 H(z)H^*(w)=(1-z/w)H^*(w)H(z),
\end{equation}
\begin{equation}\label{equ4}
 E(-z)E^*(-w)=(1-z/w)E^*(-w)E(-z),
\end{equation}
\begin{equation}\label{equ5}
 (1-z/w)H(z)E^*(-w)=E^*(-w)H(z),
\end{equation}
\begin{equation}\label{equ6}
 (1-z/w)E(-z)H^*(w)=H^*(w)E(-z).
\end{equation}

For $n\in\mathbb Z_{\geq 0}$, there are two ``extreme'' partitions of
$n$: $n$ itself and $1_n=\underbrace{(1,...,1)}_{n}$.
Set $S^n:=S^n(V)=V_n$, $\Lambda^n:=\Lambda^n(V)=V_{1_n}$.
Set furthermore
$$\mathcal H_n:=\mathcal T_{S^n} ,\; \mathcal E_n:=\mathcal T_{\Lambda^n},\;\mathcal H_n^*:=\mathcal D_{S^n},\; \mathcal E_n^*:=\mathcal D_{\Lambda^n}$$
(cf (\ref {notfunct})).
\begin{prop}\label{prop2} For $m,n\in\mathbb Z_{\geq 0}$ 
there are exact sequences of functors 
\begin{equation}\label{ex1}
0\rightarrow \mathcal H_{n}\circ\;\mathcal  H_m^*\rightarrow
\mathcal H^*_{m}\circ\; \mathcal H_{n}\rightarrow \mathcal
H^*_{m-1}\circ\; \mathcal H_{n-1}\rightarrow 0,
\end{equation}
\begin{equation}\label{ex2}
0\rightarrow \mathcal E_{n}\circ\;\mathcal
E^*_{m}\rightarrow \mathcal E^*_{m}\circ\; \mathcal
E_{n}\rightarrow \mathcal E^*_{m-1}\circ\; \mathcal
E_{n-1}\rightarrow 0,
\end{equation}
\begin{equation}\label{ex3}
0\rightarrow \mathcal H_{n}\circ\;\mathcal
E^*_{m}\rightarrow \mathcal E^*_{m}\;\circ \mathcal
H_{n}\rightarrow  \mathcal H_{n-1}\circ\;\mathcal
E^*_{m-1}\rightarrow 0,
\end{equation}
\begin{equation}\label{ex4}
0\rightarrow \mathcal E_{n}\circ\;\mathcal
H^*_{m}\rightarrow  \mathcal H^*_{m}\circ\;\mathcal
E_{n}\rightarrow  \mathcal E_{n-1}\circ\;\mathcal
H^*_{m-1}\rightarrow 0.
\end{equation}
\end{prop}

\begin{proof} We claim that there are the following exact sequences:
\begin{equation}\label{exx1}
0\rightarrow V_{m,n}\rightarrow S^m\otimes (S^n)_*\rightarrow S^{m-1}\otimes (S^{n-1})_*\rightarrow 0,
\end{equation}
\begin{equation}\label{exx2}
0\rightarrow V_{1_m,1_n}\rightarrow \Lambda^m\otimes (\Lambda^n)_*\rightarrow\Lambda^{m-1}\otimes (\Lambda^{n-1})_*\rightarrow 0,
\end{equation}
\begin{equation}\label{exx3}
0\rightarrow V_{1_m,n}\rightarrow \Lambda^m\otimes (S^n)_*\rightarrow V_{1_{m-1},n-1}\rightarrow 0,
\end{equation}
\begin{equation}\label{exx4}
0\rightarrow V_{m,1_n}\rightarrow S^m\otimes(\Lambda^n)_*\rightarrow V_{m-1,1_{n-1}}\rightarrow 0.
\end{equation}
Let us for instance construct (\ref{exx1}). By Lemma \ref{dimhom} there is a non-zero
homomorphism $\phi:S^m\otimes (S^n)_*\rightarrow S^{m-1}\otimes (S^{n-1})_*$. Clearly, 
$V_{m,n}=\soc(S^m\otimes (S^n)_*)\subset \Ker\phi$ since $V_{m,n}$ is not a
constituent of $S^{m-1}\otimes (S^{n-1})_*$ by (\ref{socmult}). Again by (\ref{socmult}),
$$[S^m\otimes (S^n)_*]=[S^{m-1}\otimes (S^{n-1})_*]+[V_{m,n}]$$
in the Grothendieck ring of $\mathbb T$, and the socle of  $(S^m\otimes (S^n)_*)/V_{m,n}$
is isomorphic to the socle of $S^{m-1}\otimes (S^{n-1})_*$. Hence $\phi$
is surjective and $V_{m,n}=\Ker\phi$.
The sequences (\ref{exx2})-(\ref{exx4}) are constructed by similar considerations.

We will now show that the sequence (\ref{exx1}) implies the existence
of (\ref{ex1}). Indeed, by the remark at the beginning of section 4, the exact sequence 
(\ref{exx1}) induces  an exact sequence
$$0\rightarrow \mathcal T_{V_{m,n}}\rightarrow\mathcal D_{S^m}\circ\; \mathcal T_{S^n}\rightarrow \mathcal D_{S^{m-1}}\circ\; \mathcal T_{S^{n-1}}\rightarrow 0.$$
Since $\mathcal T_{V_{m,n}}\simeq \mathcal T_{S^n}\circ\;\mathcal D_{S^m}$ by Proposition \ref{prop1}, the existence of 
(\ref{ex1}) follows.

The existence of sequences (\ref{ex2})-(\ref{ex4}) is proved in a
similar way by using the sequences (\ref{exx2})-(\ref{exx4}).
\end{proof}

\begin{lemma}\label{pass} For any $n\geq 0$
$$[\mathcal H_n]=h_n ,\; [\mathcal E_n]=e_n,\;[\mathcal H_n^*]=h_n^*,\; [\mathcal E_n^*]=e_n^*,$$
where $h_n,e_n,h_n^*,e_n^*$ are considered as operators on $\bold{H}^+$
as explained in Section 1. 
\end{lemma}
\begin{proof} First, observe that $[\mathcal T_N]=\operatorname{Sch}(N)$ for $N\in\TT^*$. This implies the equalities $[\mathcal H_n]=h_n , [\mathcal E_n]=e_n$.
The two other equalities follow via  Corollary \ref{coradj}. 
\end{proof}

We now see that the exact sequences (\ref{ex1})-(\ref{ex4}) categorify the identities (\ref{equ3})-(\ref{equ6}).
More precisely we have the following.
\begin{corollary}\label{cat1} \refprop{prop2} implies (\ref{equ3})-(\ref{equ6}).
\end{corollary}
\begin{proof} Consider for instance the identity (\ref{equ3})
$$H(z)H^*(w)=(1-z/w)H^*(w)H(z).$$
It is equivalent to the equality
$$h^*_m h_n-h_n h^*_m=h^*_{m-1}h_{n-1}$$
which follows immediately from (\ref{ex1}) and Lemma
\ref{pass}.

The arguments in the remaining three cases are similar. 
\end{proof}

Now we proceed to the categorification of the identities (\ref{equ1})-(\ref{equ2}).
Clearly, $H(z)$ and $E(-z)$ commute, therefore  (\ref{equ1}) is
equivalent to 
\begin{equation}\label{koz1}
H(z)E(-z)=1
\end{equation}
The equality \refeq{koz1} can be rewritten as
\begin{equation}\label{koz2}
\sum_{m+n=k}(-1)^m{h}_n{e}_m=0\;\mathrm{for\;} k>0.
\end{equation}
Similarly, (\ref{equ2}) can be rewritten as 
\begin{equation}\label{koz3}
\sum_{m+n=k}(-1)^m{h}^*_n{e}^*_m=0\;\mathrm{for\;} k>0.
\end{equation}
Lemma \ref{pass} together with \refle{lemma5} below give a categorical
proof of (\ref{koz2}) and (\ref{koz3}).

\begin{lemma}\label{lemma5}
The complexes $C_k$,
$$0\rightarrow S^k\rightarrow ...\rightarrow S^m\otimes\Lambda^n\rightarrow S^{m-1}\otimes\Lambda^{n+1}\rightarrow ...\rightarrow\Lambda^k\rightarrow 0,$$
and $(C_k)_*$,
$$0\rightarrow (S_*)^k\rightarrow \dots\rightarrow (S_*)^m\otimes(\Lambda_*)^n\rightarrow ...\rightarrow(\Lambda_*)^k\rightarrow 0,$$
are exact except for $k=0$. They induce complexes of functors
\begin{equation}\label{le5eq1}
0\rightarrow \mathcal H_k\xrightarrow{\delta} ...\xrightarrow{\delta} \mathcal H_m\circ
 \mathcal  E_{k-m}\xrightarrow{\delta} ...\xrightarrow{\delta}\mathcal E_k\rightarrow 0,
\end{equation}
\begin{equation}\label{le5eq2}
0\rightarrow \mathcal H^*_k\xrightarrow{\delta^*} ...\xrightarrow{\delta^*} \mathcal H^*_m\circ
 \mathcal E^*_{k-m}\xrightarrow{\delta^*} ...\xrightarrow{\delta^*}\mathcal E^*_k\rightarrow 0,
\end{equation}
which are exact for $k\geq 1$.
\end{lemma}

\begin{proof}
Note that $\bigoplus_{k\geq 0}C_k$ and $\bigoplus_{k\geq 0}(C_k)_*$ are Koszul complexes. 
Hence they are exact except for $k=0$. After application of the functor $\mathcal T$ we obtain that the complexes \refeq{le5eq1} and \refeq{le5eq2} are also exact
for $k\geq 1$.
\end{proof}

\section {Categorification of Clifford algebra}
Write the vertex operators $X(z)$ and $X^*(z)$ as
$$X(z)=\sum_{a\in\mathbb Z}X_az^a,\,\,X^*(z)=\sum_{a\in\mathbb Z}X^*_az^a.$$
Our next goal is to categorify the coefficients $X_a$ and $X_a^*$.
Let 
$$R_a(q):=S^{a+q}\otimes \Lambda^q_*\otimes T(V),\,\,R_a:=\bigoplus_{q\geq\operatorname{max}(-a,0)} R_a(q),$$
and
$$R_a^*(p):=\Lambda^{a+p}\otimes S_*^{p}\otimes T(V),\,\,R^*_a:=\bigoplus_{p\geq\operatorname{max}(-a,0)} R_a^*(p).$$
Define $\theta_a:R_a\to R_a$ and $\theta_a^*:R^*_a\to R^*_a$ by the formulas
$$\theta_a(x_1\dots x_p\otimes y_1\wedge\dots\wedge y_q\otimes v_1\otimes\dots\otimes v_k):=$$
$$\sum_{i,j,s}(-1)^i\langle y_i,v_j\rangle x_1\dots\hat{x}_s\dots x_p\otimes y_1\wedge\dots\wedge\hat{y_i}\wedge\dots\wedge y_q\otimes
v_1\otimes\dots\otimes v_{j-1}\otimes x_s\otimes v_{j+1}\otimes\dots\otimes v_k,$$
$$\theta_a^*(x_1\wedge\dots\wedge x_p\otimes y_1\dots y_q\otimes v_1\otimes\dots\otimes v_k):=$$
$$\sum_{i,j,s}(-1)^i\langle y_i,v_j\rangle x_1\wedge\dots\wedge\hat{x}_s\wedge\dots\wedge x_p\otimes y_1\dots\hat{y_i}\dots y_q\otimes
v_1\otimes\dots\otimes v_{j-1}\otimes x_s\otimes v_{j+1}\otimes\dots\otimes v_k.$$

Define also operators $d:R_a\to R_a$ and
$d^*:R^*_a\to R^*_a$:
$$d(x_1\dots x_p\otimes y_1\wedge\dots\wedge y_q\otimes m):=\sum_{i,j} (-1)^i \langle y_i,x_j\rangle 
 x_1\dots\hat{x}_j\dots x_p\otimes y_1\wedge\dots\wedge\hat{y_i}\wedge\dots\wedge y_q\otimes m,$$
$$d^*(x_1\wedge\dots\wedge x_p\otimes y_1\dots y_q\otimes m):=
\sum_{i,j}(-1)^j \langle y_i,x_j\rangle x_1\wedge\dots\wedge\hat{x}_j\wedge\dots\wedge x_p\otimes y_1\dots\hat{y_i}\dots y_q\otimes m.$$

\begin{lemma}\label{mainproperties}

(a) $\theta_a$, $\theta_a^*$, $d$ and $d^*$ commute with the action of the symmetric
  group $\mathbb S_k$ on $V^{\otimes k}$;

(b) $d^2=0$, $(d^*)^2=0$ and the complexes $(R_a,d)$,  $(R_a^*,d^*)$ are acyclic for
  all $a\in\mathbb Z$;

(c) $\theta_a^2=d\theta_a=-\theta_a d,\;\;(\theta_a^*)^2=d^*\theta_a=-\theta_a d^*$.
\end{lemma}
\begin{proof} (a) is straightforward, (b) is well known and also follows from (\ref{exx3})-(\ref{exx4}).
We will prove (c) by a direct calculation. Indeed, set
$$z:=x_1\dots x_p \otimes y_1\wedge\dots\wedge y_q\otimes v_1\otimes\dots\otimes v_k.$$
Furthermore, let $X_{s,t}$ stand for $x_1\dots x_p$ with $x_s$ and
$x_t$ removed, and $Y_{i,c}$ stand for $y_1\wedge\dots\wedge y_q$ with $y_c$
and $y_i$ removed and multiplied by $(-1)^{i+c}$ if $i<c$ and  $(-1)^{i+c-1}$ if $i>c$ .
Then 
$$\theta_a(\theta_a(z))=\sum_{s\neq t, j,i\neq c}\langle y_i,v_j\rangle\langle y_c,x_s\rangle X_{s,t}\otimes Y_{c,i}\otimes
v_1\otimes\dots\otimes v_{j-1}\otimes x_s\otimes v_{j+1}\otimes\dots\otimes v_k=$$
$$\theta_a\left(\sum_{c,s}(-1)^c \langle y_c,x_s\rangle x_1\dots \hat x_s\dots x_p\otimes y_1\wedge\dots\wedge\hat y_c\wedge\dots\wedge
y_q\otimes v_1\otimes\dots\otimes v_k\right)=\theta_a(d(z))=$$
$$-d\left(\sum_{i,t,j}(-1)^i \langle y_i,v_j\rangle x_1\dots \hat x_t\dots x_p\otimes y_1\wedge\dots\wedge\hat y_i\wedge\dots\wedge
y_q\otimes v_1\otimes\dots v_{j-1}\otimes x_t\otimes\dots\otimes v_k\right)=d(\theta_a(z)).$$

Checking (c) for $\theta_a^*$ and $d^*$ is similar.
\end{proof}

Let $M$ be a submodule of $T^\cdot(V)$, and $\pi_M\in\operatorname{End}_{\gg}(T^\cdot(V))$
denote a projector onto $M$.
Since $p_M$ is an element of the direct sum $\bigoplus_{n\geq o}\mathbb C[\mathbb S_n]$, Lemma \ref{mainproperties}(a) implies that 
$d, d^*$ and $\theta_a,\theta_a^*$ commute
with $\pi_M$. Therefore $d, d^*$ and $\theta_a,\theta_a^*$ 
are well defined linear operators in $R_a(M)=\pi_M R_a$ and in  $R^*_a(M)=\pi_M R^*_a$. Moreover, since any 
$M\in\mathbb T^+$ is a submodule of a direct sum of finitely many
copies of $T^\cdot(V)$, Schur--Weyl dualty implies that
$R_a(M)$, $R^*_a(M)$ and  
$d,\theta_a\in\operatorname{End}_\gg(R_a(M))$, $d^*,\theta_a^*\in \operatorname{End}_\gg(R^*_a(M))$
are well-defined for any $M\in\mathbb T^+$. 

Lemma \ref{mainproperties}(c) shows that $\Ker d$ is 
$\theta_a$-stable and  $\Ker d^*$ is $\theta_a^*$-stable. If we define
$$S_a(M):=\Ker d\cap R_a(M),\;\;S_a^*(M):=\Ker d^*\cap R_a^*(M),$$ 
and denote by $\theta_a(M)$ (respectively, $\theta_a^*(M)$) the restriction of
$\theta_a$ (respectively, $\theta_a^*$) on $S_a(M)$ (respectively, $S_a^*(M)$) then, again by  Lemma \ref{mainproperties}(c),
$(S_a(M),\theta_a(M))$  and $(S^*_a(M),\theta_a^*(M))$ become complexes. Next we set
$$\mathcal X_a(M):=(S_a(M))^+,\;\;\mathcal X_a^*(M):=(S_a^*(M))^+.$$

The exact sequences (\ref{ex3}) and (\ref{ex4}) imply $\Ker d\cap R_a(q)(M)=\mathcal T_{S^{a+q}}\circ \mathcal D_{\Lambda^{q}}(M)$ 
and $\Ker d\cap R^*_a(p)(M)=\mathcal T_{\Lambda^{a+p}}\circ \mathcal D_{S^{p}}(M)$. Therefore we have
\begin{equation}\label{vert1}
\mathcal X_a(M)=\bigoplus_{p-q=a} \mathcal H_p\circ\mathcal E^*_q(M),\,\,\mathcal X^*_a(M)=\bigoplus_{p-q=a} \mathcal E_p\circ\mathcal H^*_q(M),
\end{equation}
or simply
\begin{equation}\label{vert2}
\mathcal X_a=\bigoplus_{p-q=a} \mathcal H_p\circ\mathcal E^*_q,\,\,\mathcal X^*_a=\bigoplus_{p-q=a} \mathcal E_p\circ\mathcal H^*_q.\nonumber
\end{equation}
Moreover, $\mathcal X_a$ and $\mathcal X_a^*$ are well-defined complexes of linear endofunctors on $\mathbb T^+$: 
\begin{equation}
\mathcal X_a(q)=\mathcal H_{q+a}\circ\mathcal E^*_q,\,\,\mathcal X^*_a(q)=\mathcal E_{q+a}\circ\mathcal H^*_q,\nonumber
\end{equation}
and the differentials $\mathcal X_a(q)\to\mathcal X_a(q-1)$, $\mathcal X^*_a(q)\to\mathcal X^*_a(q-1)$ are simply the restrictions of $\theta_a$ and $\theta^*_a$
respectively. By abuse of notation we denote these differentials by the same letters $\theta_a$ and $\theta_a^*$.
These complexes can be considered as functors from $\TT^+$ to $\mathbb{DT}^+$. However, we will consider them as functors from  $\mathbb{DT}^+$ to $\mathbb{DT}^+$.
Indeed, an object of $\mathbb{DT}^+$ is isomorphic to a direct sum of simple objects $V_\lambda[n]$ for $\lambda\in\bold{Part},n\in\mathbb Z$. Then
applying $\mathcal X_a$ (or $\mathcal X^*_a$) yields a direct sum of appropriately shifted complexes. 

\begin{lemma}\label{clif1}  Let $\chi: \mathbb{DT}^+\to \bold{H}^+$
denote the Euler characteristic map. Then
$$\chi\circ\mathcal X_a=X_a\circ\chi,\;\;\chi\circ\mathcal X_a^*=X^*_a\circ\chi.$$
\end{lemma}
\begin{proof} By (\ref{vert1}) for any $M\in \TT^+$ we have 
$$\chi(\mathcal X_a(M))=\sum_{p-q=a}(-1)^q [\mathcal H_p\circ\mathcal E^*_q(M)], \,\,\chi(\mathcal X^*_a(M))=\sum_{p-q=a} (-1)^p [\mathcal E_p\circ\mathcal S^*_q(M)].$$
Therefore the statement follows from Lemma \ref{pass} and (\ref{eqvert}).
\end{proof}

Recall that by (\ref{bf}) the operators  $X_a$ and $X^*_a$ up to a shift coincide with the
generators $\psi_i,\psi^*_i$ of $\bold{Cf}$. Therefore  for an irreducible
$M\in \mathbb T^+$, we have $X_a[M]$ (or $X_a^*[M]$) is
either zero or equals the class of an irreducible object in $\mathbb T^+$ . This suggests the following.
\begin{conjecture} If $M\in\TT^+$ is irreducible, then the cohomology of each
  complex $\mathcal  X_a(M)$ and  $\mathcal  X_a^*(M)$ is nonzero in at most one degree
  and is an irreducible $\gg$-module. 
\end{conjecture}

To obtain a categorification of the Clifford generators  $\psi_i,\psi^*_i$ we use the equation (\ref {bf}). First, we 
identify $\bold F$ with the Grothendieck group of 
$\mathbb{DT}^+$ by identifying
$\mathcal K(\mathbb{T}^+[m])$ with $\bold F(m)$ for $m\in\mathbb Z$. Denote by $\mathcal{S}$ the autoequivalence $\mathbb{DT}^+\to \mathbb{DT}^+$ arising from
shifting the grading:
$\mathcal S(M):=M[1]$. Define functors $\Psi_a: \mathbb{T}[m]\to \mathbb{T}[m+1]$ and $\Psi^*_a: \mathbb{T}[m]\to \mathbb{T}[m-1]$: 
$$\Psi_a:=\mathcal S\circ\mathcal X_{a+m},\,\Psi^*_a:=\mathcal S^{-1}\circ\mathcal X_{a-m}.$$ 
Using the semisimplicity of  $\mathbb{DT}$ we extend $\Psi_a$ and $\Psi^*_a$ by additivity to functors $\mathbb{DT}^+\to \mathbb{DT}^+$. 
To check that they satisfy the Clifford 
relations it suffices to categorify the identities (\ref{v1})--(\ref{v3}).

\section{Categorifying vertex operator identities}

In this section we categorify the identities (\ref{v1})--(\ref{v3}) by using the complexes of functors
$\mathcal X_a$ and $\mathcal X_a^*$.

Let us start with (\ref{v1}).
From Proposition \ref{prop2} we have the following exact sequences of linear endofunctors on $\TT^+$
$$0\rightarrow \mathcal H_m\circ \mathcal H_p\circ\mathcal E^*_n\circ \mathcal E^*_q\rightarrow
 \mathcal H_m\circ \mathcal E^*_n\circ\mathcal H_p\circ \mathcal
 E^*_q \rightarrow  \mathcal H_m\circ \mathcal H_{p-1}\circ
   \mathcal E^*_{n-1}\circ \mathcal E^*_q\rightarrow 0,$$
$$0\rightarrow \mathcal H_{p-1}\circ \mathcal H_m\circ\mathcal
  E^*_{q}\circ \mathcal E^*_{n-1}\rightarrow\mathcal H_{p-1}\circ
   \mathcal E^*_{q}\circ \mathcal H_{m}
\circ\mathcal E^*_{n-1}\rightarrow  \mathcal H_{p-1}\circ  \mathcal H_{m-1}\circ \mathcal
E^*_{q-1}\circ \mathcal E^*_{n-1}\rightarrow 0.$$

Combining these sequences we obtain a long exact sequence
\begin{equation}\label{longexact}
\dots\rightarrow \mathcal H_{m}\circ \mathcal E^*_{n}\circ \mathcal
H_{p}\circ \mathcal E^*_{q} 
\rightarrow  \mathcal H_{p-1}\circ \mathcal E^*_{q}\circ \mathcal
H_{m}\circ \mathcal E^*_{n-1}\rightarrow  \mathcal H_{m-1}\circ
\mathcal E^*_{n-1}\circ \mathcal H_{p-1}\circ \mathcal E^*_{q-1}\rightarrow\dots.
\end{equation}
 Consider the $\mathbb Z\times\mathbb Z$-graded functor
\begin{equation}\label{bicomplexfunctor}
\mathcal X_a\circ\mathcal X_b=\bigoplus_{n,q} \mathcal X_a(n)\circ\mathcal X_b(q),\nonumber
\end{equation}
where here and below $\mathcal K(n)$ stands for the $n$-th term of a complex of endofunctors $\mathcal K$.
Let
$$\chi(\mathcal X_a\circ\mathcal X_b)=\sum_{n,q}(-1)^{n+q} [\mathcal X_a(n)]\circ[\mathcal X_b(q)].$$
We rewrite  (\ref{longexact}) in the form
\begin{equation}\label{longexact1}
\dots\to \mathcal X_a(n)\circ\mathcal X_b(q)\to  \mathcal X_{b-1}(q)\circ\mathcal X_{a+1}(n-1)\to  \mathcal X_a(n-1)\circ\mathcal X_b(q-1)\to\cdots.
\end{equation}
Then (\ref{longexact1}) implies 
$$\chi(\mathcal X_a\circ\mathcal X_b)+\chi(\mathcal X_{b-1}\circ\mathcal X_{a+1})=0$$ 
which yields (\ref{v1}). The proof of  (\ref{v2}) is similar.

Note that we can define two morthisms of functors : 
$$\theta'_{ab}:\mathcal X_a(n)\circ\mathcal X_b(q)\to \mathcal X_a(n)\circ\mathcal X_b(q-1),$$
$$\theta''_{ab}:\mathcal X_a(n)\circ\mathcal X_b(q)\to \mathcal X_a(n-1)\circ\mathcal X_b(q)$$
by setting 
$$\theta'_{ab}\left(\mathcal X_a(n)\circ\mathcal X_b(q)(M)\right):=\mathcal X_a(n)(\theta_b( X_b(q)(M))),$$
$$\theta''_{ab}\left(\mathcal X_a(n)\circ\mathcal X_b(q)(M)\right):=\theta_a(\mathcal X_a(n)( X_b(q)(M)))$$
for every $M\in\TT^+$. It is easy to check that $(\theta'_{ab})^2=(\theta''_{ab})^2=0$, and moreover
we have the following.
\begin{conjecture}\label{comrel}
(a) $\mathcal X_a\circ\mathcal X_b$ is a bicomplex of functors with differentials  $\theta'_{ab},\theta''_{ab}$.

(b) (\ref{longexact1}) is the cone of a map of total complexes $\mathcal X_a\circ\mathcal X_b\to  \mathcal X_{b-1}\circ\mathcal X_{a+1}$ .

(c) Analogous statements hold for $\mathcal X_a^*\circ\mathcal X_b^*$.
\end{conjecture}

It remains to prove (\ref{v3}). Define the functors  $\mathcal Z_{a,b}(n,q):\TT^+\to \TT^+$ by setting
$$\mathcal Z_{a,b}(n,q):=\mathcal H_{a+n}\circ  \mathcal E_{b+q}\circ \mathcal E^*_{n}\circ \mathcal H^*_{q}.$$
We rewrite (\ref{ex1}) and (\ref{ex2}) in the form
\begin{equation}\label{equ11}
0\to \mathcal Z_{a,b}(n,q)\to \mathcal X_a(n)\circ\mathcal X^*_b(q)\to \mathcal X_{a+1}(n-1)\circ\mathcal X^*_{b-1}(q)\to 0,
\end{equation}
\begin{equation}\label{equ22}
0\to \mathcal Z_{a,b}(n,q)\to \mathcal X^*_b(n)\circ\mathcal X_a(q)\to \mathcal X^*_{b+1}(q-1)\circ\mathcal X_{a-1}(n)\to 0.
\end{equation}

\begin{conjecture} $\mathcal Z_{a,b}=\bigoplus_{n,q}\mathcal Z_{a,b}(n,q)$ has the structure of a bicomplex,
and the morphisms in (\ref{equ11}) and (\ref{equ22}) commute with the differentials.
\end{conjecture}

Set
\begin{equation}\label{bicomplexfunctor1}
\mathcal X_a\circ\mathcal X^*_b=\bigoplus_{n,q} \mathcal X_a(n)\circ\mathcal X^*_b(q),\,\,
\mathcal X_b^*\circ\mathcal X_a=\bigoplus_{n,q} \mathcal X^*_b(q)\circ\mathcal X_a(n).\nonumber
\end{equation}

Observe that for any fixed $M\in\TT^+$ we have  $\mathcal Z_{a,b}(M)=0$ for
sufficiently large negative $b$.
Therefore (\ref{equ11}) implies
$$\mathcal X_a\circ\mathcal X^*_b\simeq\bigoplus_{i>0}\mathcal Z_{a+i,b-i}.$$
Similarly by  (\ref{equ22})  we have
$$\mathcal X^*_b\circ\mathcal X_a\simeq\bigoplus_{i<0}\mathcal Z_{a+i,b-i}.$$
Hence 
\begin{equation}\label{sum}
\mathcal X_a\circ\mathcal X^*_b\oplus \mathcal
X^*_{b+1}\circ\mathcal X_{a-1}=\bigoplus_{i\in\mathbb Z}\mathcal Z_{a+i,b-i}.\nonumber
\end{equation}

Set $$\chi( \mathcal Z_{a,b})=\sum(-1)^{n+q}[ \mathcal Z_{a,b}(n,q)].$$

\begin{lemma}\label{zchar}
$$\sum_{i\in\mathbb Z}\chi(\mathcal Z_{a+i,b-i})=\delta_{a+b,0}\operatorname{Id}.$$
\end{lemma}
\begin{proof} Using the complex $\mathcal C_k$ (see Lemma \ref{lemma5}) we obtain
$$\sum_{m+p=k} (-1)^p [\mathcal H_m\circ \mathcal E_p]=\delta_{k,0}\operatorname{Id},$$
and from the complex $(\mathcal C_l)_*$ we obtain
$$\sum_{n+q=l} (-1)^q  [\mathcal E^*_n\circ \mathcal H^*_q]=\delta_{l,0}\operatorname{Id}.$$
Combining the both above identities we have
$$\sum_{m+p=k,n+q=l} (-1)^{p+q} [\mathcal H_m\circ  \mathcal E_p\circ  \mathcal E^*_{n}\circ \mathcal H^*_{q}]=\delta_{k,0}\delta_{l,0}\operatorname{Id}.$$
After summation over all $k-l=s$ we get
$$\sum_{m+p-n-q=s} (-1)^{p+q} [\mathcal H_m\circ  \mathcal E_p\circ \mathcal E^*_n\circ \mathcal H^*_{q}]=\delta_{s,0}\operatorname{Id}.$$
The left-hand side of the last identity is equal to
$\sum_{i\in\mathbb Z}\chi(\mathcal Z_{s+i,-i})$.
Hence the statement.
\end{proof}
\begin{corollary} We have
$$\chi(\mathcal X_a\circ\mathcal X^*_b)+\chi(\mathcal  X^*_{b+1}\circ\mathcal X_{a-1})=\delta_{a+b,0}\operatorname{Id}.$$
Hence (\ref{v3}) holds.
\end{corollary}

\section{Categorification of Weyl algebra}
We will now construct  certain graded endofunctors $\mathcal P_n,\mathcal P^*_n$ categorifying
$p_n,p_n^*$, and will use   $\mathcal P_n,\mathcal P^*_n$ to show that $p_n$ and $p^*_n$ satisfy the relations
(\ref{boson}). Recall that the diagrams $(p,1_q)$ are called {\it hooks}.
Define $\mathcal P_k(i):=\mathcal T_{V_{k-i,1_i}}$, $\mathcal P_k(i):=\mathcal D_{V_{k-i,1_i}}$ and set
$$\mathcal P_k:=\bigoplus_{i=0}^{k-1}\mathcal P_k(i),\,\,\,\mathcal P^*_k:=\bigoplus_{i=0}^{k-1}\mathcal P^*_k(i).$$

Consider again the complexes of functors (\ref{le5eq1}) and
(\ref{le5eq2}) with differentials $\delta$ and $\delta^*$.

\begin{lemma}\label{hookkoszul} We have the following exact sequence of functors
$$0\to\mathcal P_k(i)\to\mathcal H_{k-i}\circ\mathcal E_i\xrightarrow{\delta}\mathcal H_{k-i-1}\circ\mathcal E_{i+1},$$ 
and
$$0\to\mathcal P^*_k(i)\to\mathcal H^*_{k-i}\circ\mathcal E^*_i\xrightarrow{\delta}\mathcal H^*_{k-i-1}\circ\mathcal E^*_{i+1}.$$
\end{lemma}
\begin{proof} The first exact sequence follows from the following well-known exact sequence (see, for instance, \cite{FH}, Exercise 6.20)
$$0\to V_{p,1_q}\to S^p\otimes \Lambda^{q-1}\to  S^{p-1}\otimes \Lambda^{q}.$$
The second one follows from the exact sequence
$$0\to (V_*)_{p,1_q}\to S_*^p\otimes \Lambda_*^{q-1}\to  S_*^{p-1}\otimes \Lambda_*^{q}.$$
\end{proof}

Lemma \ref{hookkoszul} motivates us to define $\mathcal P_k$ and $\mathcal P^*_k$ as complexes of endofunctors in $\TT^+$ with zero differentials.
As in the fermion case we will consider $\mathcal P_k$ and $\mathcal P^*_k$ as functors from  $\mathbb{DT}^+$ to $\mathbb{DT}^+$.

\begin{lemma}\label{cath}
$$\chi\circ\mathcal P_k=p_k\circ\chi,\;\;\chi\circ\mathcal P^*_k=p^*_k\circ\chi.$$
\end{lemma} 
\begin{proof} The well-know identity \cite{M} in $\bold{H}^+$
$$p_k=\sum_{i=0}^{k-1} (-1)^i s_{k-i,1_i}$$
implies 
$$\chi\circ\mathcal P_k=\sum_{i=0}^{k-1}(-1)^i[\mathcal P_k(i)]=\sum_{i=0}^{k-1} (-1)^i[\mathcal T_{V_{k-i,1_i}}]=p_k\circ\chi.$$
The second equality is similar.
\end{proof}

Now we are going to use $\mathcal P_n$ and $\mathcal P_n^*$ in order
to show that $p_n$ and $p^*_n$ satisfy the relations
(\ref{boson}). 
The equality $[p_n,p_m]=0$ follows directly from the
commutativity of the tensor product, and the equality  $[p^*_n,p^*_m]=0$ follows from
Lemma \ref{lemma3} and the commutativity of the tensor product. So it
remains to prove that
$$[p_m^*,p_k]=k\delta_{m,k}.$$
At this time we do not have a truly categorical proof of the last identity. Instead we will pass to the Grothendieck ring of
$\TT^+$.
Note that by Proposition \ref{prop1} it is sufficient to prove the
following identity in $\mathcal K(\TT^+)$: 
\begin{equation}\label{hookid}
\sum_{i=0}^{k-1}\sum_{j=0}^{m-1}(-1)^{i+j}[(V_{(k-i,1_i)}\otimes
  (V_{(m-j,1_j)})_*/V_{(k-i,1_i),(m-j,1_j)}]=k\delta_{m,k}[\mathbb C].\nonumber
\end{equation}

We are going to use (\ref{socmult}). If $\lambda$ is a hook, then
$N^\lambda_{\gamma,\nu}\leq  1$. Moreover, $N^\lambda_{\gamma,\nu}=1$
implies that $\gamma$ and $\nu$ are also hooks. If
$\lambda=(c,1_d)$, $\nu=(p,1_q)$ and $\gamma=(s,1_t)$, then $N^\lambda_{\gamma,\nu}=1$
if and only if $s+p=c,t+q=d$ or $s+p=c-1,t+q=d+1$.
If $\gamma=(p,1_q)$ we set $s(\gamma)=(-1)^q$. The above implies that for any $\nu$ such that $|\nu|< m$ and for any hook $\gamma$
\begin{equation}\label{auxp}
\sum_{j=0}^{m-1}(-1)^j N^{m-j,1_j}_{\nu,\gamma}=(-1)^{s(\gamma)}\delta_{\nu,\emptyset}\delta_{|\gamma|,m}.
\end{equation}
Now by  (\ref{socmult}) we have 
$$\sum_{i=0}^{k-1}\sum_{j=0}^{m-1}(-1)^{i+j}[(V_{(k-i,1_i)}\otimes
  (V_{(m-j,1_j)})_*)/V_{(k-i,1_i)),(m-j,1_j)}:V_{\mu,\nu}]=\sum_{i=0}^{k-1}\sum_{j=0}^{m-1}\sum_{\gamma}(-1)^{i+j}N^{(k-i,1_i)}_{\mu,\gamma}N^{(m-j,1_j)}_{\nu,\gamma}.$$
By (\ref{auxp}) we obtain
$$\sum_{i=0}^{k-1}\sum_{j=0}^{m-1}\sum_{\gamma}(-1)^{i+j}N^{(k-i,1_i)}_{\mu,\gamma}N^{(m-j,1_j)}_{\nu,\gamma}=
\sum_{\gamma}\left(\sum_{i=0}^{k-1}(-1)^i N^{m-i,1_i}_{\mu,\gamma}\right)\left(\sum_{j=0}^{m-1}(-1)^j N^{m-j,1_j}_{\nu,\gamma}\right)=
\sum_{\gamma}\delta_{\mu,\emptyset}\delta_{|\gamma|,k}\delta_{\nu,\emptyset}\delta_{|\gamma|,m}=k\delta_{k,m}\delta_{\nu,\emptyset}\delta_{\mu,\emptyset}.$$

\end{document}